\documentclass[12pt]{amsart}

\usepackage{amsfonts, amsmath, amscd}
\usepackage[psamsfonts]{amssymb}

\usepackage{amssymb}

\usepackage[colorlinks]{hyperref} 
\hypersetup{ linkcolor=blue }

\usepackage[usenames]{color}

\newtheorem{theorem}{Theorem}[section]
\newtheorem{lemma}[theorem]{Lemma}
\newtheorem{corollary}[theorem]{Corollary}
\newtheorem{question}[theorem]{Question}
\newtheorem{remark}[theorem]{Remark}
\newtheorem{proposition}[theorem]{Proposition}

\newtheorem{example}[theorem]{Example}
\newtheorem{fact}[theorem]{Fact}
\newtheorem{claim}[theorem]{Claim}

\newtheorem{conjecture}[theorem]{Conjecture}

\numberwithin{equation}{section}

\newcommand{\NN}{\mathbb{N}}
\newcommand{\CC}{\mathcal{C}}
\newcommand{\DD}{\mathcal{D}}

\newcommand{\BB}{\mathcal{B}}
\newcommand{\UU}{\mathcal{U}}
\newcommand{\WW}{\mathcal W}

\newcommand{\EE}{{\mathcal E}}

\newcommand{\FF}{\mathcal{F}}

\newcommand{\RR}{\mathbb{R}}

\newcommand{\Nn}{\mathcal{N}}
\newcommand{\PP}{\mathbb{P}}

\newcommand{\AAA}{\mathcal A}

\newcommand{\cantor}{2^{\omega}}
\newcommand{\weakstar}{\ensuremath{weak^*}}
\newcommand{\Int}{\operatorname{int}}
\newcommand{\diam}{\operatorname{diam}}
\newcommand{\eps}{\varepsilon}
\renewcommand{\phi}{\varphi}
\newcommand{\Q}{\mathbb Q}
\newcommand{\ball}{\mathbf B}
\newcommand{\sphere}{\mathbf S}

\title[Networks for the weak topology 0f Banach and Fr\'echet spaces]{Networks for the weak topology of Banach and Fr\'echet spaces}

\author{S. Gabriyelyan}
\thanks{The first named author was partially supported by Israel Science Foundation grant 1/12.}
\address{Department of Mathematics, Ben-Gurion University of the Negev,
Beer-Sheva P.O. 653, Israel}
\email{saak@math.bgu.ac.il}

\author{J. K{\c{a}}kol}
\thanks{The second named author was supported by Generalitat
Valenciana, Conselleria d'Educaci\'{o}, Cultura i Esport, Spain, Grant
PROMETEO/2013/058. Part of this work has been done when the second named  author visited Department of Mathematics Czech Academy of Science as a Visiting Professor, April 2014.}
\address{Faculty of Mathematics and Informatics A. Mickiewicz University, $61-614$ Pozna{\'n}, Poland.}
\email{kakol@amu.edu.pl}

\author{W. Kubi\'s}
\thanks{The third named author was supported by GA\v CR grant No. P201 14 07880S}
\address{Institute of Mathematics,
Academy of Sciences of the Czech Republic, \v Zitn\'a 25, 115 67 Praha 1, Czech Republic}
\email{kubis@math.cas.cz}

\author{W. Marciszewski}
\thanks{The fourth named author was supported by National Science Center research grant DEC-2012/07/B/ST1/03363}
\address{Institute of Mathematics, University of Warsaw, Banacha 2, 02--097 Warszawa, Poland}
\email{wmarcisz@mimuw.edu.pl}

\dedicatory{Dedicated to the Memory of
Professor Manuel Valdivia}


\begin{document}

\begin{abstract}
We start the systematic study of Fr\'{e}chet spaces which are $\aleph$-spaces in the weak topology. A topological space $X$ is an $\aleph_0$-space or an $\aleph$-space if $X$ has a countable $k$-network or  a $\sigma$-locally finite $k$-network, respectively. We are motivated by the following result of Corson (1966): If the space $C_{c}(X)$ of continuous real-valued functions on a Tychonoff space $X$ endowed with the  compact-open topology is a Banach space, then $C_{c}(X)$ endowed  with the weak topology is an $\aleph_0$-space if and only if $X$ is countable. We extend Corson's result as follows: If the space $E:=C_{c}(X)$ is a Fr\'echet lcs, then  $E$ endowed with its weak topology $\sigma(E,E')$  is an  $\aleph$-space  if and only if $(E,\sigma(E,E'))$ is an $\aleph_0$-space if and only if $X$ is countable. We obtain a necessary and some sufficient conditions on a Fr\'echet lcs to be an $\aleph$-space in the weak topology. We prove that a reflexive Fr\'echet lcs $E$  in the weak topology  $\sigma(E,E')$ is  an $\aleph$-space  if and only if  $(E,\sigma(E,E'))$ is an $\aleph_0$-space if and only if $E$ is separable. We show however that the nonseparable Banach space $\ell_{1}(\mathbb{R})$   with the weak topology is an $\aleph$-space.
\end{abstract}

\maketitle

\section{Introduction}

Topological properties of a  locally convex space (lcs for short) $E$ in the weak topology $\sigma(E,E')$ are of the importance and have been intensively studied from many years (see  \cite{kak,bonet}).
Corson (1961) started a systematic study of certain topological properties of the weak topology of Banach spaces. This line of research provided more general classes such as reflexive Banach spaces, weakly compactly generated Banach spaces ($WCG$ Banach spaces)  and the class of weakly $K$-analytic and weakly $K$-countably determined Banach spaces. We refer the reader to \cite{fabian} and \cite{kak} for many references and facts.

Although  $(E,\sigma(E,E'))$ is never a  metrizable space  for a separable infinite dimensional normed $E$, every $\sigma(E,E')$-compact set is $\sigma(E,E')$-metrizable (see \cite[Corollary 4.6]{kak} or \cite[Proposition 3.29]{fabian}).
Moreover, for many natural and important classes of separable metrizable lcs $E$, the space $(E,\sigma(E,E'))$ is a generalized metric space of some type (see \cite{BBK,GKKLP,Mich}).
Such types of topological spaces are defined by different types of networks.
The concept of network is one of a well recognized good tool, coming from the pure set-topology, which turned out to be of great importance to study successfully renorming theory in Banach spaces, see the survey paper \cite{cascales}; especially  \cite[Theorem 13]{cascales} for $\sigma(E,F)$-slicely networks.

Following Michael \cite{Mich}, a family $\Nn$ of subsets of a topological space $X$ is called a {\em $k$-network in $X$} if whenever $K\subset U$ with $K$ compact and $U$ open in $X$, then $K\subset \bigcup\FF\subset U$ for some finite $\FF\subset\Nn$. A topological space $X$ is said to be an {\em $\aleph_0$-space} if $X$ is regular and  has a countable $k$-network \cite{Mich}.
It is known that a regular space is an $\aleph_{0}$-space if and only if it is a continuous image of a separable metric space under a compact-covering mapping (\cite{Mich}). Every $\aleph_{0}$-space is  separable and Lindel\"of.
It is known that every Banach space  $E$ whose strong dual $E'$  is separable is a {\it weakly $\aleph_{0}$-space}, i.e. $E$ with the weak topology $\sigma(E,E')$ is an $\aleph_{0}$-space,  see \cite{Mich} and  \cite{BBK} (or \cite{GKKLP} for more general facts  for \emph{Fr\'echet lcs}, i.e. metrizable and complete lcs).

 O'Meara \cite{OMe2} generalized  the concept of $\aleph_0$-spaces as follows: A topological space $X$ is called an {\em $\aleph$-space} if it is regular and has a $\sigma$-locally finite $k$-network. Any metrizable space is an $\aleph$-space and all compact sets in $\aleph$-spaces are metrizable. For further results, see \cite{gruenhage}.
The  study of  those locally convex  spaces $E$ which are  weakly $\aleph$-spaces (i.e. under the weak topology $E$ has a $\sigma$-locally finite $k$-network) is begun here for the important particular case of Fr\'echet lcs.

If $X$ is a Tychonoff space, $C_c(X)$ (resp. $C_p(X)$) denotes the space $C(X)$ of all continuous functions on $X$ endowed with the compact-open (resp. pointwise) topology. It is well known that $C_{c}(X)$ is metrizable if and only if $X$ is \emph{hemicompact}, i.e. $X$ admits a fundamental sequence of compact sets, see \cite{arens}.
Moreover, $C_{c}(X)$ is complete if and only if $X$ is a $k_{R}$-space, see \cite{bonet}.
Note that $C_p(X)$ is an $\aleph$-space if and only if $X$ is countable \cite{Sak}. Corson \cite{Mich} proved the following interesting result (*): If $K$ is  compact, the Banach space $C(K)$ is a weakly  $\aleph_0$-space if and only if $K$ is countable.
Our main result extends Corson's  theorem.
\begin{theorem} \label{tMain}
A Fr\'echet lcs  $C_{c}(X)$ is a weakly  $\aleph$-space if and only if $C_{c}(X)$ is a weakly $\aleph_{0}$-space if and only if $X$ is countable.
\end{theorem}
If $X$ is a countable and {\it locally compact} space, Theorem \ref{tMain} guarantees that $C_c(X)$ is even a weakly  $\aleph_0$-space. We note the following question: Is $C_c(X)$ a weakly  $\aleph$-space for any countable Tychonoff space $X$? Having in mind  that the weak topology of $C_{c}(X)$ lies  between the compact open topology and the pointwise one, the question is especially interesting for the case $X$ is an $\aleph_0$-space. Recall that for such $X$  the spaces  $C_c(X)$ and $C_{p}(X)$ are  $\aleph_0$-spaces by \cite{Mich}, and   $C_p(X)$  is even separable and metrizable. In Section \ref{sCount} we prove that $C_c(X)$ is a weakly  $\aleph_0$-space for any countable $\aleph_0$-space $X$.

Although  $\aleph$-spaces  and $\aleph_{0}$-spaces are essentially different, in the class of Lindel\"of spaces they coincide, see Proposition \ref{pLindel} below. Therefore, it is interesting also to  describe  possible large classes of Fr\'{e}chet (or Banach) spaces for which the both concepts coincide for the weak topology.
We observe that any $WCG$  Banach space is a weakly $\aleph$-space if and only if it is a weakly $\aleph_{0}$-space (Corollary \ref{tal}).  We show   that a Banach space $E$ not containing a copy of $\ell_1$ is a weakly $\aleph$-space if and only if it is a weakly $\aleph_{0}$-space if and only if the strong dual $E'$ of $E$ is separable  (Corollary \ref{c-Equiv}).  This extends a corresponding result for the case $\aleph_{0}$-spaces (see  \cite[\S 12]{BBK} and \cite{GKKLP}). Consequently, for any  $1<p<\infty$ and an uncountable  set $\Gamma$ the (reflexive) Banach space $\ell_{p}(\Gamma)$ is not a weakly $\aleph$-space. We show even more: A reflexive Fr\'echet lcs $E$ is a weakly $\aleph$-space if and only if $E$ is a weakly $\aleph_{0}$-space if and only if $E$ is separable (Corollary \ref{fre}). These results motivate the following  natural question: \textit{Does there exist a nonseparable Banach space $E$ which is an $\aleph$-space in the weak topology of $E$?} We answer this question in the affirmative by proving the following theorem.

\begin{theorem}\label{ThmElloneAlephCont}
The Banach space $\ell_1(\Gamma)$ is an $\aleph$-space in the weak topology if and only if the cardinality of $\Gamma$ does not exceed the continuum.
\end{theorem}
So, the nonseparable Banach space $\ell_1(\mathbb{R})$ endowed with the weak topology is an $\aleph$-space  but is not an $\aleph_0$-space.  Moreover, the space $\ell_1(\mathbb{R})$ in the weak topology is not normal, see Proposition \ref{exa:foged}.


\section{Some definitions and known facts}


Recall (see \cite{Gao}) that a family $\Nn$ of subsets of a topological space $X$  is a {\em $cs^\ast$-network   at a point $x\in X$} if for each sequence $(x_n)_{n\in\NN}$ in $X$ converging to  $x$ and for each neighborhood $O_x$ of $x$ there is a set $N\in\mathcal{N}$ such that $x\in N\subset O_x$ and the set $\{n\in\NN :x_n\in N\}$ is infinite (where $\NN =\{ 1,2,\dots \}$); $\Nn$ is a {\em $cs^\ast$-network}  in $X$ if $\mathcal{N}$ is a $cs^\ast$-network at each point $x\in X$.
The smallest size $|\mathcal{N}|$ of a $cs^\ast$-network $\mathcal{N}$ at $x$ is called the {\em $cs^\ast$-character of $X$ at the point $x$} and is denoted by $cs^\ast_\chi(X,x)$. The cardinal $cs^\ast_\chi(X)=\sup\{ cs^\ast_\chi(X,x): x\in X\}$ is called the {\em $cs^\ast$-character} of  $X$.
Recall also (see \cite{Mich}) that a point $x$ in a topological space $X$ is called an {\it $r$-point} if there is a sequence $\{ U_n\}_{n\in\NN}$ of neighborhoods of $x$ such that if $x_n\in U_n$, then $\{ x_n\}_{n\in\NN}$ has compact closure; call $X$ an {\it $r$-space} if all of its points are $r$-points. The first countable spaces and the locally compact spaces are $r$-spaces.
\begin{theorem}[\cite{OMe}] \label{fOMe}
A topological space $X$ is metrizable if and only if it is an $\aleph$-space and an $r$-space.
\end{theorem}
A topological space $X$ has the {\it property $\left( \alpha_{4}\right) $ at a point $x\in X$} if for any $\{x_{m,n}:\left( m,n\right) \in \mathbb{N}\times \mathbb{N}\}\subset X$ with $\lim_{n}x_{m,n}=x\in X$, $m\in \mathbb{N}$, there exists a sequence $\left( m_{k}\right) _{k}$ of distinct natural numbers and a sequence $\left( n_{k}\right) _{k}$ of natural numbers such that $\lim_{k}x_{m_{k},n_{k}}=x$; $X$ has the {\it property $\left( \alpha_{4}\right) $} or is an {\it $\left( \alpha_{4}\right) $-space} if it has the property $\left( \alpha_{4}\right)$ at each point $x\in X$. Nyikos proved in \cite[Theorem 4]{nyikos} that any Fr\'{e}chet-Urysohn topological group satisfies $\left( \alpha_{4}\right)$. However there are Fr\'{e}chet-Urysohn topological spaces which do not have $\left( \alpha_{4}\right)$ (see for instance Example \ref{exaFUMet}).

 Theorem \ref{fOMe} combined with additional facts from  \cite{BZ} yields also
\begin{theorem} \label{tMetr-Aleph}
An $\aleph$-space $X$  is metrizable if and only if $X$ is a Fr\'{e}chet-Urysohn  $\left( \alpha_{4}\right)$-space.
\end{theorem}
\begin{proof}
Clearly, if $X$  is metrizable then it is a Fr\'{e}chet-Urysohn $\left( \alpha_{4}\right)$-space. Conversely, let $X$ be  a Fr\'{e}chet-Urysohn  $\left( \alpha_{4}\right)$-space. Being an $\aleph$-space, $X$ has countable  $cs^\ast$-character (this might be also noticed from the proof of  \cite[Corollary 2.18]{Sak}). Indeed, it immediately follows from the definitions of $k$- and $cs^\ast$-networks that any closed $k$-network is a $cs^\ast$-network. So
it is enough to show that any space $X$ with a $\sigma$-locally finite closed $cs^\ast$-network $\DD =\bigcup_{n\in\NN} \DD_n$ has countable $cs^\ast$-character. Fix $x\in X$. For every $n\in\NN$ set $T_n(x):=\{ D\in\DD_n : x\in D\}$. Since $\DD_n$ is locally finite, the family $T_n(x)$ is finite. So the family $T(x):= \bigcup_{n\in\NN} T_n(x)$ is countable. We show that $T(x)$ is a countable $cs^\ast$-network at $x$. Let $x_n \to x$ and $U$ be a neighborhood of $x$.
Since $\DD$ is a $cs^\ast$-network, there is $k\in\NN$ and $D\in\DD_{k}$ such that $x \in D\subset U$ and $D$ contains infinitely many elements of $\{ x_n\}_{n\in\NN}$.
As $D$ is closed, it contains $x$, so $D\in T_{k}(x)$. Now  \cite[Proposition 6, Lemma 7]{BZ} imply that $X$ is first countable, hence an $r$-space. Finally, $X$ is metrizable by Theorem \ref{fOMe}.
\end{proof}
Since every Fr\'{e}chet-Urysohn topological group satisfies property $(\alpha_4)$ by \cite[Theorem 4]{nyikos},
we obtain
\begin{corollary}
A topological group $G$ is metrizable if and only if $G$ is a Fr\'{e}chet-Urysohn $\aleph$-space.
\end{corollary}

\begin{example} \label{exaFUMet} {\em
Let $V(\aleph_0)$ be the Fr\'{e}chet-Urysohn fan which is obtained by the identifying the limit points of the compact spaces $S_n$, where $S_n = \left\{ 0, \frac{1}{k}\right\}_{k\in\NN} \subset \mathbb{R}$. It is well-known that any compact subset of $V(\aleph_0)$ is contained in a finite union of sequences $S_n$. Hence the natural map from the topological direct sum $\bigoplus_{k\in\NN} S_k$ onto $V(\aleph_0)$ is compact-covering. So $V(\aleph_0)$ is an $\aleph_0$-space \cite{Mich}. Clearly, $V(\aleph_0)$ is not metrizable and not  an $\left( \alpha_{4}\right) $-space (and it  is not an $r$-space by Theorem \ref{fOMe}).}
\end{example}

\section{Some necessary conditions for being an $\aleph$-space}

Recall that a topological space $X$ is called a {\it $\sigma$-space} if $X$ is  regular and  has a $\sigma$-discrete (equivalently, $\sigma$-locally finite) network. If $X$ is regular and has a countable network, $X$ is called a {\it cosmic space}. Clearly  $\aleph$-spaces and cosmic spaces are  $\sigma$-spaces.
It is well known (see \cite{gruenhage}) that any closed subset $H$ of a $\sigma$-space $X$ is a $G_\delta$-set. Indeed, if $\bigcup_{n\in\NN} \DD_n$ is a $\sigma$-discrete closed network for $X$, then the sets $A_n :=\bigcup\{ D\in\DD_n : D\cap H =\emptyset \}$ are closed in $X$ by \cite[1.1.11]{Eng}. As $H=\bigcap_{n\in\NN} (X\setminus A_n)$,  $H$ is a $G_\delta$-set. Consequently, any $\sigma$-space has  countable pseudocharacter;  we denote $\psi(X)=\aleph_0$.

Clearly, every separable Banach space with the Schur property is a weakly $\aleph_{0}$-space.
In Section \ref{SecWitek} we show that $\ell_{1}(\RR)$ is a weakly $\aleph$-space.

It is well known that the dual space of $\ell_p(\Gamma)$ is $\ell_q(\Gamma)$, where $1/p + 1/q =1$. So,  the support of continuous
functionals over  $\ell_p(\Gamma)$ must be countable. We use this fact to prove the following
\begin{example} \label{lp}
Let $\Gamma$ be an infinite set and $E:=\ell_p(\Gamma)$ with  $1<p<\infty$. Then $\psi(E_w)\geq|\Gamma|$, where $E_{w}:=(E,\sigma(E,E'))$.  Hence  $\ell_{p}(\Gamma)$ are not  weakly $\sigma$-spaces for every uncountable $\Gamma$.
\end{example}
\begin{proof}
If $\Gamma$ is countable the assertion is clear. Suppose that $\Gamma$ is uncountable.
Let $\mathcal{U}=\{U_i\}_{i\in I}$ be a family of weakly open neighborhoods of $0$ such that $\bigcap_{i\in I} U_i =\{ 0\}$ and $|I|=\psi(E_w)$. We may assume  that each  $U_i$ has the following standard form
\[
U_i = \{ x\in E: | \chi_{i,k} (x)| <\delta_i, \mbox{ where } \chi_{i,k} \in E' \mbox{ for } 1\leq k \leq m_i \}.
\]
Suppose, for a contradiction, that $|I|< |\Gamma|$. Denote by $J$ the set of all indices $j\in I$ such that the $j$-coordinate is nonzero for some $\chi_{i,k}$. So $|J|=|\NN|\times |\NN| \times |I| < |\Gamma|$. Hence we can find an index $\gamma_0 \in \Gamma\setminus J$. Set $x_0=(r_\gamma)_{\gamma\in \Gamma}$, where $r_\gamma =1$ if $\gamma=\gamma_0$, and $r_\gamma =0$ otherwise. Clearly, $x_0\in E$ and $\chi_{i,k} (x_0) =0$ for all $i\in I$ and every $1\leq k \leq m_i$, that contradicts the choice of the family $\mathcal{U}$. Thus $|I|\geq |\Gamma|$.
\end{proof}

 We  provide some necessary condition for any lcs to be  a weakly $\aleph$-space.
First we prove the following useful observation.
\begin{lemma} \label{l-PsChar}
Let $E$ be a non-trivial lcs. Then $E_w:=(E,\sigma(E,E'))$ has countable pseudocharacter if and only if $E_w$ admits a weaker separable metrizable lcs topology. In particular, $|E|=\mathfrak{c}$ provided $E_{w}$ has countable pseudocharacter.
\end{lemma}

\begin{proof}
Assume that $E_w$ has countable pseudocharacter.
Let $\bigcap_{n\in\NN} U_n = \{0\}$, where the open sets $U_n$ have the following standard form
\[
U_n = \{ x\in E: | \chi_{i,n} (x)| <\delta_n, \mbox{ where } \chi_{i,n} \in E' \mbox{ for } 1\leq i \leq k_n \}.
\]
Let $\{ \chi_n\}_{n\in\NN}$ be an enumeration of the family $\{ \chi_{i,n}: 1\leq i \leq k_n, n\in\NN\}$.  Then  $\bigcap_{n\in\NN} \ker(\chi_n) =\{ 0\}$. This implies that the following map
\[
p: E_w \to \prod_{n\in\NN} E/\ker(\chi_n) =\mathbb{R}^\NN, \quad p(x)=\left( \chi_n (x)\right)_{n\in\NN},
\]
is continuous and injective, and hence $|E|=\mathfrak{c}$ as $E$ is non-trivial. Now the topology induced on $E_w$ from $\mathbb{R}^\NN$ is as desired. The converse assertion is trivial.
\end{proof}
Next fact is well known but hard to locate.
\begin{lemma} \label{propDual}
A lcs $E$ admits a metrizable and separable locally convex topology $\tau$ weaker than $\sigma(E,E')$ if and only if $(E',\sigma(E',E))$ is separable.
\end{lemma}




Lemmas \ref{l-PsChar} and \ref{propDual} imply the following necessary conditions on lcs $E$ to be
weakly $\aleph$-space which partially converse Proposition \ref{pro}(i).
\begin{proposition} \label{p-Nec}
If $E$ is a non-trivial lcs which is a  weakly $\sigma$-space, then
\begin{enumerate}
\item[{\rm (i)}] $(E, \sigma(E,E'))$ admits a weaker separable metrizable lcs topology;
\item[{\rm (ii)}] $\psi (E, \sigma(E,E')) = \aleph_0$ and $|E|=\mathfrak{c}$;
\item[{\rm (iii)}] $(E',\sigma(E',E))$ is separable.
\end{enumerate}
\end{proposition}
Note that the space $\ell_\infty$ satisfies above  conditions (i)--(iii), although it is not a weakly $\aleph$-space (see Corollary \ref{ell} below).


\section{$\ell_1(\RR)$ is an $\aleph$-space in the weak topology}\label{SecWitek}


In this section we prove Theorem \ref{ThmElloneAlephCont} which states that the Banach space $\ell_1(\Gamma)$ is an $\aleph$-space in the weak topology if and only if the cardinality of $\Gamma$ does not exceed the continuum. In particular, the space $\ell_{1}(\RR)$ is a weakly $\aleph$-space.
Clearly, the ``only if" part of the theorem follows from Proposition \ref{p-Nec} because the space $\ell_1(\Gamma)$ with the weak topology does not have countable pseudocharacter whenever $|\Gamma| > 2^{\aleph_0}$.
The remaining part of this section is devoted to the proof of the ``if" part.  It is clear that if the cardinality of the set $\Gamma_1$ is less than or equal to the cardinality of the set $\Gamma_2$ then $\ell_1(\Gamma_1)$ embeds into $\ell_1(\Gamma_2)$, therefore it is enough to consider the case when $\Gamma$ has the cardinality continuum.

We shall work with the space $\ell_1(\cantor)$, where $\cantor$ denotes the Cantor set, treated just as an index set of cardinality continuum (recall that the space $\ell_1(S)$ does not depend on any extra structure of the set $S$).

We shall use some ideas from \cite{MarPol} (especially from the proof of Lemma 2.3.1 in \cite{MarPol}).
Given a Banach space $E$, we shall denote by $\ball_E$ and $\sphere_E$ the closed unit ball and the unit sphere of $E$, respectively.

\begin{lemma}
The unit sphere $\sphere_{\ell_\infty(\cantor)}$ is \weakstar\ separable.
\end{lemma}

\begin{proof}
Let $\PP$ be the family of all (necessarily finite) partitions of the Cantor set into finitely many open sets.
As $\cantor$ is zero-dimensional, for every finite set $F \subset \cantor$ there is $P \in \PP$ such that $P = \{U_x \colon  x\in F\}$ and $x \in U_x$ for every $x \in F$.
Obviously, the family $\PP$ is countable, because the Cantor set has only countably many sets that are open and closed simultaneously.
Define
\[
D = \left\{ \sum_{U \in P} q_U \chi_U \in S_{\ell_\infty(\cantor)} \colon P \in \PP, \; \{ q_U \colon U \in P\} \subset \Q \right\},
\]
where $\chi_A$ denotes the characteristic function of a set $A$. Obviously, $D$ is countable. We claim that it is \weakstar\ dense in $\sphere_{\ell_\infty(\cantor)}$.

In fact, given $x_1, \dots, x_k \in \ell_1(\cantor)$ and $\eps > 0$, a basic \weakstar\ neighborhood of $y \in S_{\ell_\infty(\cantor)}$ is of the form
\[
V = \{ v \in S_{\ell_\infty(\cantor)} \colon |v(x_i) - y(x_i)| < \eps \text{ for }i=1,2,\dots,k \}.
\]
Fix $\delta > 0$ and let $F \subset \cantor$ be a finite set such that
\begin{equation} \label{equation-1}
\|x_i\| - \sum_{t \in F} |x_i(t)| = \sum_{t \not\in F} |x_i(t)| <\delta
\end{equation}
for every $i =1,2,\dots, k$.
Take a partition $P \in \PP$ such that $U \cap F$ is either empty or a singleton, whenever $U \in \PP$, and there is $U\in P$ such that $U\cap F=\emptyset$. For every $t\in F$ and each $U\in P$ containing $t\in U$ take $q_U\in [-1,1]\cap \Q$ such that $|q_U - y(t)| < \delta$, and set $q_U =1$ for every $U\in P$ such that $U\cap F=\emptyset$. Set $w = \sum_{U \in P} q_U \chi_U$. Then $w\in D$. We show that $w\in y+V$ for $\delta$ small enough. Indeed, for every $i =1,2,\dots, k$, the inequality (\ref{equation-1}) and the construction of $w$ imply
\[
\begin{split}
|w(x_i) -y(x_i)| & \leq \sum_{t\in F} |w(t) x_i(t) - y(t) x_i(t)| +\sum_{t\not\in F} |w(t) x_i(t) - y(t) x_i(t)| \\
 & < \delta \cdot \sum_{t\in F} |x_i(t)| + \sum_{t\not\in F} 2 | x_i(t)| < \delta(\| x_i\| +2).
\end{split}
\]
Now it is clear that if $\delta$ is small enough then $w\in y+V$. Thus $\sphere_{\ell_\infty(\cantor)}$ is \weakstar\ separable.
\end{proof}

\begin{lemma}\label{Lmerbgui3}
Let $E$ be a Banach space such that $(\sphere_{E'}, \weakstar)$ is separable.
Then for every $r>0$ there exists a countable family $\FF$ of weakly closed subsets of $E$ contained in $$E \setminus r \ball_E = \{ x \in E \colon \|x\| > r\}$$
and such that $$E \setminus r \ball_E = \bigcup_{F \in \FF} \Int_w(F),$$ where $\Int_w$ denotes the interior with respect to the weak topology.
\end{lemma}

\begin{proof}
Let $D$ be a countable \weakstar\ dense subset of $\sphere_{E'}$.
Given $\phi \in D$, $n \in \NN$, define
\[
F_{\phi,n} = \{ x \in E \colon \phi(x) \geq r + 1/n \}.
\]
Then $\FF = \{F_{\phi,n} \colon \phi \in D, \; n \in \NN\}$ is the required family.
\end{proof}

\begin{lemma}\label{Lmsgribwri}
Let $0 < \eps < r$ and let
\[
M(r,\eps) = \{ x \in \ell_1(\cantor) \colon r - \eps < \|x\| \leq r \}.
\] 
Then for every $x \in M(r,\eps)$ there exists a weakly (in fact, pointwise) open set $V \subset \ell_1(\cantor)$ such that $x \in V$ and $\diam(V \cap M(r,\eps)) \leq 4 \eps$.
\end{lemma}

\begin{proof}
Given $A \subset \cantor$, denote by $p_A$ the canonical projection from $\ell_1(\cantor)$ onto $\ell_1(A)$, that is, $p_A(v) = v \restriction A$.
Fix $x \in M(r,\eps)$.
There exists a finite set $F \subset \cantor$ such that $\| p_F(x) \| > r - \eps$.
Choose an open set $U \subset \ell_1(F)$ such that $\|u\| > r-\eps$ and $\|u - p_F(x)\| < \eps$ for every $u \in U$.
Let $V = p_F^{-1}(U)$.
We claim that $V$ is as required.

Obviously, $x \in V$ and $V$ is pointwise (in particular, weakly) open. Fix $y_1, y_2 \in V \cap M(r,\eps)$.
Let $A = \cantor \setminus F$.

Note that the $\ell_1$-norm has the property that
\[
\| v \| = \| p_F(v) \| + \|p_A(v) \|
\]
for every $v \in \ell_1(\cantor).$ In particular, $\| p_A(y_i) \| \leq \eps$, because $\|p_F(y_i)\| > r-\eps$ and $\|y_i\| \leq r$ for $i=1,2$.
Using these facts we get
\begin{align*}
\|y_1 - y_2\| &= \| p_F(y_1) - p_F(y_2) \| + \| p_A(y_1) - p_A(y_2) \| \\
&\leq \| p_F(y_1) - p_F(x) \| + \| p_F(x) - p_F(y_2) \| + \|p_A(y_1)\| + \|p_A(y_2)\| \\
& \leq 4 \eps.
\end{align*}
It follows that $\diam(V \cap M(r,\eps)) \leq 4\eps$.
\end{proof}

The next statement is rather standard; it has been used implicitly, e.g., in \cite{MarPol}.

\begin{lemma}\label{Lmebrwefeox}
Let $X$ be a metric space. Then there exists an open base $\BB$ in $X$ such that $\BB = \bigcup_{n \in \NN}\BB_n$ and each $\BB_n$ is uniformly discrete, that is, for every $n$ there is $\eps_n>0$ such that the distance of any two distinct members of $\BB_n$ is $>\eps_n$.
\end{lemma}

\begin{proof}
A theorem of Stone says that every open cover of a metric space $X$ admits a $\sigma$-discrete open refinement.
The proof (see, e.g., \cite[Proof of Thm. 4.4.1]{Eng}) actually shows that every open cover of $X$ has an open refinement of the form $\UU = \bigcup_{n\in\NN}\UU_n$, where each $\UU_n$ is uniformly discrete.
Now let $\BB = \bigcup_{n \in \NN}\WW_n$ be such that $\WW_n$ is an open refinement of a cover by balls of radius $1/n$ and $\WW_n$ is a countable union of uniformly discrete families.
Then $\BB$ is easily seen to be an open base.
\end{proof}

\begin{remark} \label{rem:Schur}{\em
The proof of Theorem \ref{ThmElloneAlephCont} uses the well known fact stating that the space $\ell_1(\Gamma)$ has the Schur property (that is any convergent sequence in the weak topology is also a  convergent sequence in the norm topology)  for every set $\Gamma$, see  \cite{fabian}. This implies that any weakly compact set of $\ell_1(\Gamma)$ is also norm compact.}
\end{remark}

\begin{proof}[Proof of Theorem \ref{ThmElloneAlephCont}]
Let $\BB = \bigcup_{n \in \NN}\BB_n$ be a base of open sets for the norm topology on $\ell_1(\cantor)$ such that the distance between every two distinct members of $\BB_n$ is $>1/k_n$ for every $n \in \NN$ (here we have used Lemma~\ref{Lmebrwefeox}).

Given $r>0$, define
\[
U_r = \ell_1(\cantor) \setminus r \ball_{\ell_1(\cantor)}.
\]
Let $\FF_r$ be a countable family of weakly closed subsets of $U_r$ such that $\bigcup_{F \in \FF_r} \Int_w F = U_r$ (Lemma~\ref{Lmerbgui3}). Let $\FF_r = \{F^m_r\}_{m \in \NN}$.

Given $n,m,i \in \NN $, define
\[
L(i,n):= M\left( \frac{i+2}{10k_n}, \frac{1}{5k_n}\right) = \left\{ x \in \ell_1(\cantor) \colon \frac{i}{10k_n} < \|x\| \leq \frac{i+2}{10k_n} \right\}
\]
and
\[
\CC(n,m,i) = \{ B \cap F^m_{i/10k_n} \cap L(i,n) \colon B \in \BB_n \}.
\]

\begin{claim}
For every $n,m,i \in \NN$, the family $\CC(n,m,i)$ is discrete in the weak topology.
\end{claim}

\begin{proof}
Note that the union of $\CC(n,m,i)$ is contained in the weakly closed set 
\[
F^m_{i/10k_n} \cap ((i+2)/10k_n)\ball_{\ell_1(\cantor)},
\]
therefore it is enough to show that every point of this set has a weak neighborhood meeting at most one set from $\CC(n,m,i)$.

Fix $x \in F^m_{i/10k_n} \cap ((i+2)/10k_n)\ball_{\ell_1(\cantor)} \subset L(i,n)$. By Lemma~\ref{Lmsgribwri}, there exists a weakly open set $V$ such that $x \in V$ and
\[
\diam\big(V \cap L(i,n) \big) \leq 4/5k_n < 1/k_n.
\]
The set $V$ can intersect at most one $B \in \BB_n$, as $\BB_n$ is $1/k_n$-discrete.
\end{proof}

Let $O(r) = \{ x \in \ell_1(\cantor) \colon \|x\| < r\}$ and define
\[
\DD(n) = \{ B \cap  O(1/5k_n) \colon B \in \BB_n \}.
\]
Note that actually $\DD(n)$ contains at most one nonempty set (because $2/5k_n < 1/k_n$), therefore it is certainly discrete in the weak topology.

Define
\[
\AAA = \bigcup \{ \CC(n,m,i) \colon n,m,i \in \NN \} \cup \{ \DD(n) \colon n \in \NN \}.
\]

Then the family $\AAA$ is $\sigma$-discrete with respect to the weak topology. It remains to show that $\AAA$ is a $k$-network in $\ell_1(\cantor)$ with the weak topology.

Fix a weakly compact set $K$ contained in a weakly open set $U \subset \ell_1(\cantor)$. It follows from Remark \ref{rem:Schur} that $K$ is also compact in the norm topology. Choose a finite $\EE \subset \BB$ such that $K \subset \bigcup \EE \subset U$. Choose $n_0$ such that $\EE \subset \bigcup_{n=1}^{n_0}\BB_{n}$.

For $i,n\in\NN$, let $N(i,n) := \Int\big( L(i,n)\big)$. Note that, for each $n\in\NN$, the space $\ell_1(\cantor)$  is covered by $O(1/5k_n)$ and the sets $N(i,n)$, $i\in \NN$. Given $B \in \BB_n$, by Lemma~\ref{Lmerbgui3} we have that
\[
B = \bigcup \left\{ B \cap \Int_w(F^m_{i/10k_n}) \cap N(i,n) \colon m,i \in \NN \right\} \cup(B\cap O(1/5k_n)) .
\]
Therefore the family
\begin{eqnarray*}
\bigcup_{n\leq n_0} \left( \left\{ B \cap \Int_w(F^m_{i/10k_n}) \cap N(i,n) \colon B \in \BB_n \cap \EE, \; m,i \in \NN \right\} \right.\\ \cup \left. \{B\cap O(1/5k_n)\colon B \in \BB_n \cap \EE\}\right)
\end{eqnarray*}
covers $K$ and consists of norm open sets, so it has a finite subfamily covering $K$. 
Hence, for any $n\leq n_0$, we can find $i_0(n)$ and $m_0(n)$ such that the finite subfamily
\begin{eqnarray*}
\FF := \bigcup_{n\leq n_0} \left(\left\{ B \cap F^m_{i/10k_n}  \cap L(i,n) \colon B \in \BB_n \cap \EE, \; m\leq m_0(n), \;  i \leq i_0(n)\right\}\right. \\ \cup  \{B\cap O(1/5k_n)\colon B \in \BB_n \cap \EE\}\big)
\end{eqnarray*}
of $\AAA$ satisfies $K\subset \bigcup\FF \subset \bigcup\EE \subset U$. Thus the $\sigma$-locally finite family $\AAA$ is also a $k$-network for $\ell_1(\cantor)$ in the weak topology.
\end{proof}

\begin{proposition} \label{exa:foged}
The space $\ell_{1}(\mathbb{R})$ endowed with the weak topology is not normal.
\end{proposition}
\begin{proof}
Suppose for a contradiction that $\ell_{1}(\mathbb{R})$ with the weak topology is a normal space. Then the square $\ell_{1}^{2}(\mathbb{R})$ of $\ell_{1}(\mathbb{R})$ with the weak topology $\omega$  is also normal (note that  $\ell_{1}^{2}(\mathbb{R})$ and $\ell_{1}(\mathbb{R})$ endowed with the weak topologies are homeomorphic). Now  Corson's lemma  \cite[Lemma 7]{corson} applies to derive that every $\omega$-discrete set in  $\ell_{1}(\mathbb{R})$ is countable, which clearly leads to a contradiction.
\end{proof}

\begin{remark} {\em
Recall  that Foged in \cite{foged} constructed already a non-normal space which is an $\aleph$-space.  Our example of such a space seems to be however very natural and uses a well known Banach space $\ell_{1}(\mathbb{R})$.  Also  O'Meara \cite{OMe1} gave an example (unpublished) of an $\aleph$-space which is not paracompact.
The authors thank to Professor  Gary Gruenhage for providing references included in the above remark.}
\end{remark}

Notice also that the proof of Theorem \ref{ThmElloneAlephCont} essentially uses the fact that $\ell_1(\mathbb{R})$ has the Schur property (see Remark \ref{rem:Schur}). Therefore it is natural to ask: 
\begin{question}
Let $E$ be a Banach space with the Schur property and satisfy (i)-(iii) of Proposition \ref{p-Nec}. Is $E$ an $\aleph$-space in the weak topology?
\end{question}
Taking into account Remark  \ref{rem:Schur}, every separable Banach space with the Schur property in the weak topology is an $\aleph_{0}$-space.


\section{Interplay between weakly $\aleph$ and weakly $\aleph_0$-Fr\'{e}chet spaces}


Recall that a lcs $E$ is called {\it trans-separable} if for each neighborhood of zero $U$ in $E$ there exists a
countable subset $N$ of $E$ such that $E=N+U$. Clearly for metrizable lcs trans-separability and separability are equivalent concepts.

\begin{lemma}[{\cite[Cor. 6.8]{kak}}] \label{sep}
The strong dual of a lcs $E$ is trans-separable if and only if every bounded set in $E$ is metrizable in the weak topology  of $E$.
\end{lemma}

Recall that a  Fr\'{e}chet lcs $E$ satisfies the  \emph{density condition} if every bounded set in $E'$ (with the strong topology) is metrizable (cf. \cite[Prop. 6.16]{kak}).
The class of such spaces includes Fr\'{e}chet-Montel locally convex spaces and  quasinormable Fr\'{e}chet locally convex spaces. The latter class contains  all Banach spaces, as well as every $(FS)$-space (see  \cite{bierstedt2}).
In \cite{GKKLP} we proved the following
\begin{proposition}[\cite{GKKLP}]  \label{pro}
Let $E$ be a Fr\'echet lcs and $E'$ be its strong
dual. Then
\begin{itemize}
\item[\textrm{(i)}] If $E'$ is separable, then $E$ is a weakly $\aleph_{0}$-space.
\item[\textrm{(ii)}] If $E$ is a weakly $\aleph_{0}$-space not containing a copy of $\ell_{1}$, then $E'$ is trans-separable.
\item[\textrm{(iii)}] If $E$ is a weakly $\aleph_{0}$-space, then $E'$ is trans-separable if and only if every bounded set in $E$ is Fr\'echet-Urysohn in the weak topology of $E$.
\item[\textrm{(iv)}] If $E$ satisfies the density condition and does not contain a copy of $\ell_{1}$, then  $E$ is a weakly $\aleph_{0}$-space  if and only if $E'$ is separable.
\item[\textrm{(v)}] If $E$ does not contain a copy of $\ell_{1}$, then every bounded set in $E$ is Fr\'echet-Urysohn in $\sigma(E,E')$.
\end{itemize}
\end{proposition}

\begin{corollary}\label{fre}
A reflexive Fr\'echet lcs $E$  is a weakly $\aleph$-space if and only if $E$ is separable (if and only if $E$ is a weakly $\aleph_{0}$-space).
\end{corollary}
\begin{proof}
As $E$ is reflexive, $(E',\sigma(E',E))$ is separable if and only if $(E',\beta(E',E))$ is separable. Assume that $E$ is a weakly $\aleph$-space. Then $(E',\sigma(E',E))$ is separable by Proposition \ref{p-Nec}, so  $(E',\beta(E',E))$ is separable. By Proposition \ref{pro}(i) the space $E$ is a weakly $\aleph_{0}$-space. In particular, $E$ is separable. Conversely, if $E$ is separable then $(E',\sigma(E',E))$ is separable and Proposition \ref{pro} (i)  applies.
\end{proof}
Since every nuclear Fr\'{e}chet space is a separable reflexive space, see \cite{bierstedt2},  we have
\begin{corollary} \label{cNuclear}
Every nuclear Fr\'{e}chet space is a weakly $\aleph_{0}$-space.
\end{corollary}

We apply  Theorem \ref{tMetr-Aleph}  to extend parts  (ii) and (iii) of Proposition \ref{pro}.
\begin{theorem} \label{dual-Aleph}
Let $E$ be a lcs which is a weakly $\aleph$-space. Then the strong dual $E'$ of $E$ is trans-separable if and only if every bounded set in $E$ is Fr\'echet-Urysohn in the weak topology of $E$.
If in addition $E$ is a Fr\'echet lcs not containing a copy of $\ell_{1}$,  then  $E'$  is trans-separable.
\end{theorem}
\begin{proof}
If $E'$ is trans-separable, then every bounded set in $E$ is metrizable in $\sigma(E,E^{\prime })$ by Lemma \ref{sep}. Conversely, if every bounded set in $E$ is Fr\'echet-Urysohn in $\sigma(E,E')$, apply  \cite[Lemma 3.2]{GKKLP} to see that every bounded set $B$ in $E$ is a Fr\'echet-Urysohn $(\alpha_4)$-space in $\sigma(E,E')$. As a subspace of the $\aleph$-space $(E,\sigma(E,E'))$,  $B$ is also an $\aleph$-space. By Theorem \ref{tMetr-Aleph}, $B$ is metrizable. Finally, Lemma \ref{sep} applies to get the trans-separability of $E'$.
The last assertion follows from the first one and Proposition \ref{pro}(v).
\end{proof}
As the strong dual of a Banach space is normed, this theorem combined with Proposition \ref{pro} yield the following
\begin{corollary} \label{c-Equiv}
Let $E$ be a Banach space not containing a copy of $\ell_{1}$. Then $E$ is a weakly $\aleph$-space if and only if  $E$ is a weakly $\aleph_0$-space if and only its strong dual  $E'$ is separable.
\end{corollary}

Any reflexive Fr\'echet lcs $E$ does not contain  a copy of $\ell_{1}$, but $E$ may not satisfy the density condition \cite{bierstedt1}.
The following result  generalizes  (iv) of Proposition \ref{pro}.
\begin{theorem} \label{tDensity}
Let $E$ be a Fr\'echet lcs not containing a copy of $\ell_{1}$ and  satisfying the density condition. Then $E$ is a weakly $\aleph$-space if and only if  $E$ is a weakly $\aleph_0$-space if and only if the strong dual of $E'$ of $E$ is separable.
\end{theorem}

\begin{proof}
Clearly, the strong dual $E^{\prime }$ is a $(DF)$-space, see \cite[Theorem 8.3.9]{bonet}, with a fundamental sequence $(Q_{n})_{n}$ of absolutely convex bounded subsets of $E^{\prime }$. Since $E$ satisfies the density condition, every bounded set $Q_{n}$ is
metrizable by \cite[Corollary 3]{bierstedt2}.
Assume now that $E$ is a weakly $\aleph$-space. By Theorem \ref{dual-Aleph} the strong dual $E'$ is trans-separable. So the trans-separable  lcs $E'$ is covered by a sequence of metrizable bounded absolutely convex sets $(Q_{n})_{n}$. Now Corollary 4.12 of \cite{GKKLP} implies that $E'$ is separable. As  $E'$ is separable, then $E$ is a weakly $\aleph_0$-space by Proposition \ref{pro}(i). Finally, if $E$ is a weakly $\aleph_0$-space it is also a weakly $\aleph$-space.
\end{proof}

Since every Fr\'echet lcs $C_{c}(X)$  satisfies the density condition (see \cite{peris} or \cite{bierstedt2}),  we apply Theorem \ref{tDensity} to get
\begin{corollary} \label{quasi}
Let $E:=C_{c}(X)$ be a Fr\'echet lcs not containing a copy of $\ell_{1}$. Then $E$ is a weakly $\aleph$-space if and only if $E$ is a weakly $\aleph_{0}$-space if and only if the strong dual $E'$ of $E$ is separable.
\end{corollary}

We need the following useful fact, see also \cite{OMe2} for (ii).
\begin{proposition} \label{pLindel}
Let $X$ be a topological space.
\begin{enumerate}
\item[{\rm (i)}] $X$ is a cosmic space if and only if $X$ is a Lindel\"{o}f $\sigma$-space.
\item[{\rm (ii)}]  $X$ is an $\aleph_0$-space if and only if $X$ is a Lindel\"{o}f $\aleph$-space.
\end{enumerate}
\end{proposition}
\begin{proof}
Assume that $X$ is a Lindel\"of  $\sigma$-space (respectively, an $\aleph$-space)  with a $\sigma$-locally finite network (respectively, $k$-network) $\DD=\bigcup_{n\in\NN} \DD_n$. It is enough to prove that every $\DD_n$ is countable. For every $x\in X$ choose an open neighborhood $U_x$ of $x$ such that $U_x$ intersects with a finite subfamily $T(x)$ of $\DD_n$. Since $X$ is a Lindel\"{o}f space, we can find a countable set $\{ x_k\}_{k\in\NN}$ in $X$ such that $X =\bigcup_{k\in\NN} U_{x_k}$. Hence any $D\in\DD_n$ intersects with some $U_{x_k}$ and therefore $D\in T(x_k)$. Thus $\DD_n = \bigcup_{k\in\NN} T(x_k)$ is countable.

Conversely, if $X$ is a cosmic (respectively, an  $\aleph_0$-space), then $X$ is Lindel\"{o}f (see \cite{Mich}) and it is trivially a  $\sigma$-space (respectively, an   $\aleph$-space).
\end{proof}
\begin{corollary}\label{final}
Let $E$ be a Lindel\"{o}f (in particular, separable metrizable) lcs. Then  $E$ is a weakly $\aleph$-space if and only if $E$ is  a weakly  $\aleph_0$-space.
\end{corollary}

Since every $WCG$ Banach  space is Lindel\"{o}f in its weak topology by Preiss-Talagrand's theorem
(see \cite[Theorem 12.35]{fabian}), we note also
\begin{corollary}\label{tal}
Every $WCG$ Banach space is a weakly $\aleph$-space if and only if it is a weakly $\aleph_{0}$-space.
\end{corollary}

As $C_c(X)$ is Lindel\"{o}f for any $\aleph_0$-space $X$ by \cite[Proposition 10.3]{Mich}, we obtain

\begin{corollary} \label{mi} Let $X$ be an $\aleph_0$-space. Then $C_c(X)$ is a weakly $\aleph$-space if and only it is a weakly $\aleph_0$-space.
\end{corollary}

\section{Proof of Theorem \ref{tMain}}

We need the following  lemmas.
\begin{lemma} \label{l-Normal}
Let $X$ be a completely regular space containing a non-scattered compact subset $K$. Then $C_c(X)$ is not a weakly $\aleph$-space.
\end{lemma}
\begin{proof}
Suppose for a contradiction that $C_c(X)$ is  a weakly $\aleph$-space. As $K$ is not scattered, there exists a continuous map $f$ from $K$ onto the interval $[0,1]$ (see \cite[Theorem 8.5.4]{sema}).
In particular, every compact subset of $[0,1]$ is the image $f(L)$ for some compact set $L$ in $X$.
By the Tietze-Urysohn theorem (which holds for compact subsets of completely regular spaces, knowing that they have normal compactifications), $f$ has an extension ${\tilde f}: X\to [0,1]$. Clearly, ${\tilde f}$ is also compact-covering, therefore the adjoint map $h \mapsto h\circ{\tilde f}$ is an embedding of $C[0,1]$ into $C_c(X)$.
Finally, if $C[0,1]$ were an $\aleph$-space in the weak topology, then by Corollary~\ref{mi} it would be an $\aleph_0$-space, which leads to a contradiction with the following result of Corson (see~\cite[Prop. 10.8]{Mich}): A space of the form $C(K)$ with $K$ compact is an $\aleph_0$-space in the weak topology if and only if $K$ is countable.
\end{proof}

For example, as $X=\NN^\NN$ has a non-scattered compact subset, the space $C_c(X)$ is an $\aleph_0$-space \cite{Mich}, but $C_c(X)$ is not  a weakly $\aleph_0$-space by Lemma \ref{l-Normal}.
Observe that  the condition on $X$ to have only scattered (even countable) compact subsets is not enough for $C_c(X)$ to be a weakly $\aleph$-space. This follows from the following
\begin{lemma} \label{2-Normal}
Let $X$ be a Tychonoff space such that each compact subset of $X$ is countable. If $C_c(X)$ is a weakly $\aleph$-space, then $X$ is separable. In particular, the space $C_c[0,\omega_1)$ is not a weakly $\aleph$-space.
\end{lemma}
\begin{proof}
By Proposition \ref{p-Nec}, there is a sequence $\{ K_n\}_{n\in\NN}$ of (countable) compact subsets of $X$ and a sequence $\{ \delta_n \}_{n\in\NN}$ of positive numbers such that
\begin{equation} \label{e-2}
\bigcap_{n\in\NN} \{ f\in C_c(X): \; f(K_n)\subseteq [-\delta_n, \delta_n] \} =\{ 0\}.
\end{equation}
Set $A:=\cup_{n\in\NN} K_n$. Then $A$ is countable. We show that $A$ is dense in $X$. Indeed, if $X\setminus \mathrm{cl}_X (A) \not= \emptyset$, we can find $h\not= 0$ such that $h(\mathrm{cl}_X (A))=\{ 0\}$, that contradicts (\ref{e-2}).
The last assertion follows from the fact that $[0,\omega_1)$  is a non-separable locally compact normal space  (see \cite[3.1.27]{Eng}).
\end{proof}


\begin{lemma} \label{3-Normal}
Let $X$ be a completely regular space. Then the following assertions are equivalent:
\begin{enumerate}
\item[{\rm (i)}]  $X$ contains a non-scattered compact subset.
\item[{\rm (ii)}] $C_c(X)$ contains a copy of $\ell_1$.
\item[{\rm (iii)}] $C_c(X)$ contains a separable Banach space $B$ with non-separable dual.
\end{enumerate}
So, every compact subset of $X$ is scattered if and only if $C_c(X)$ does not contain a copy of $\ell_1$.
\end{lemma}

\begin{proof}
(i)$\Rightarrow$(ii) Assume that $X$ contains a non-scattered compact set $K$. As it was shown in the proof of Lemma \ref{l-Normal}, the space $C[0,1]$ embeds into $C_c(X)$. It remains to note that $C[0,1]$ contains $\ell_1$.

(ii)$\Rightarrow$(iii) is trivial. Let us prove  that
(iii)$\Rightarrow$(i) Suppose for a contradiction that every compact subset of $X$ is scattered. Denote by $\mathcal{K}$ the set of all compact subsets of $X$. Then  $C_c(X)$ can be treated as a subspace of the product $E:= \prod_{K\in \mathcal{K}} C_c(K)$. Then, by \cite[Theorem 4.1 and the first claim of the proof]{DMS}, the space $B$ is embedded in the finite product $F_{k}:=C(K_{i_{1}})\times\dots\times C(K_{i_{k}})$ for some $k\in\mathbb{N}$. On the other hand, since $F_{k}=C(K_{i_{1}}\oplus\dots\oplus K_{i_{k}})$ and the topological direct sum $K:=K_{i_{1}}\oplus\dots\oplus K_{i_{k}}$ is compact and scattered,
$F_k$ is Asplund by \cite[Theorem 12.29]{fabian} (i.e. every separable subspace of $C(K)$ has separable dual).  Thus $F_k$ does not contain $B$, a contradiction.
\end{proof}

We recall that the countable product of $\aleph_{0}$-spaces is an $\aleph_{0}$-space (see \cite{Mich}).  Now we are ready to prove the main theorem.
\begin{proof}[Proof of Theorem \ref{tMain}]
Let $(K_{n})_{n}$ be a fundamental sequence of compact sets in $X$. Then $C_{c}(X)$ is embedded in the product $\prod_{n}C(K_{n})$.

If $X$ is countable, each space $K_{n}$ is metrizable and scattered, so  $C_{c}(X)$ is a weakly $\aleph_0$-space by Proposition \ref{pro}(i). Thus $C_c(X)$ is a weakly $\aleph$-space.

Assume that $C_c(X)$ is a weakly $\aleph$-space. By Lemma \ref{l-Normal}, $K_{n}$ is scattered for every $n\in\NN$.
We apply  Lemma \ref{3-Normal} to derive that $C_{c}(X)$ does not  contain a copy of $\ell_{1}$.
Now Corollary  \ref{quasi} says that $C_{c}(X)$ is a weakly $\aleph_{0}$-space.

Assume now that $C_c(X)$ is a weakly $\aleph_0$-space. Lemma \ref{l-Normal}  shows that every compact subset of $X$ is scattered.  Since $C_{c}(X)$ is separable,  $X$ admits a weaker metrizable topology; so all sets $K_{n}$ are metrizable and scattered. Thus each $K_{n}$ is countable, so it is the whole space $X$.
\end{proof}
Theorem \ref{tMain} combined with \cite{MarPol} provides concrete Banach spaces $C(K)$ which under the weak topology are $\sigma$-spaces but is not $\aleph$-spaces.
\begin{corollary}\label{mar}
Let $K$ be an uncountable separable compact space. If $K$ is

(1) a linearly ordered space, or

(2) a dyadic space,

then $C(K)$ endowed with the weak topology is a $\sigma$-space but not an $\aleph$-space.
If  additionally $K$ is metrizable, then  $C(K)$ endowed with the weak topology is a  cosmic space but is not an $\aleph$-space.
\end{corollary}
\begin{proof} For the both cases, by Theorem \ref{tMain}, the space $C(K)$ with the weak topology is not an $\aleph$-space.

(1): Let $K$ be a compact space as assumed. By  \cite[Theorem 5.5]{MarPol} the space $(C(K),\tau_{p})$ is a $\sigma$-space, where $\tau_{p}$ is the pointwise topology on $C(K)$.  Moreover,  by Lemma 5.4 and Lemma 2.3.1 of \cite{MarPol} the  space  $C(K)$ admits a $\sigma$-discrete collection in $(C(K),\tau_{p})$ which is a network in $C(K)$. Hence $C(K)$ with the weak topology is a $\sigma$-space.

(2):  By \cite[Lemma 2.3.1 and Lemma 5.10]{MarPol} there exists a $\sigma$-discrete family in $(C(2^{2^{\omega}}),\tau_{p})$ which is a network in $C(2^{2^{\omega}})$. Hence the space $C(2^{2^{\omega}})$ endowed with the weak topology  is a $\sigma$-space. Since $K$ is a continuous image of $2^{2^{\omega}}$, the space $C(K)$ embeds into $C(2^{2^{\omega}})$ for the weak topology. This proves the general case.

Assume now that $K$ is metrizable. Then $C(K)$ is a Polish space. So $C(K)$ endowed with the weak topology is a  cosmic space by \cite{Mich} but is not  an $\aleph$-space by Theorem \ref{tMain}.
\end{proof}

\begin{remark} {\em
Let $K$ be a compact space. Then $C(K)$ is weakly cosmic if and only if $C(K)$ is separable if and only if $K$ is metrizable. So, if $K$ satisfies (1) or (2) of Corollary \ref{mar} but is not metrizable (for example, $K$ is the Bohr compactification of a countably infinite abelian group; in this case $K$ satisfies (2) but is not metrizable) we obtain a non-trivial example of Banach spaces $B$ such that $(B,\sigma(B,B'))$ is a $\sigma$-space but $(B,\sigma(B,B'))$ is neither cosmic nor an $\aleph$-space.}
\end{remark}

\begin{corollary}\label{fre}
Let $X$ be a locally compact and paracompact space. Then $C_{c}(X)$ is a weakly $\aleph$-space if and only if $X$ is countable.
\end{corollary}
\begin{proof}
By the assumption on $X$, there exists a family $\{X_{t}:t\in T\}$ of locally compact and $\sigma$-compact spaces such that $X:=\bigoplus_{t\in T}X_{t}$, see \cite[5.1.27]{Eng}. So $C_{c}(X)=\prod_{t\in T}C_{c}(X_{t})$ and each $C_{c}(X_{t})$ is a Fr\'{e}chet space.

Assume that $C_{c}(X)$ is a weakly $\aleph$-space. Then $C_{c}(X_{t})$ is a weakly $\aleph$-space and hence $X_t$ is countable  by Theorem \ref{tMain} for every $t\in T$. As any compact subset of $C_{c}(X)$ is metrizable by Proposition \ref{p-Nec}(i), the set $T$ is countable. Thus $X$ is countable. Conversely, if $X$ is countable, $C_{c}(X)$ is a Fr\'{e}chet space and  Theorem \ref{tMain} applies.
\end{proof}
Since $\ell_{\infty}=C(\beta\mathbb{N})$, Theorem \ref{tMain} provides
\begin{corollary}\label{ell}
A Banach space containing a copy of $\ell_{\infty}$ is not a weakly $\aleph$-space.
\end{corollary}

\begin{corollary}
Let $X$ be a locally compact Hausdorff space. Then the Banach space $C_{0}(X)$ of continuous functions on $X$ vanishing at infinity is a weakly $\aleph$-space if and only if $X$ is countable.
\end{corollary}
\begin{proof}
Let $K$ be the  one-point compactification of $X$. Then $C(K) = C_{0}(X) \oplus \RR$, therefore Theorem \ref{tMain} applies.
\end{proof}

\begin{question}\label{qu}
Let $X$ be a Tychonoff space. Is it true that $X$ is countable provided that $C_{c}(X)$ is a weakly $\aleph$-space?
\end{question}


\section{$C_c(X)$ over a countable $\aleph_0$-space $X$ for the weak topology} \label{sCount}

This section is motivated by the previous one; especially by Question  \ref{qu}. Let $\mu$ be a ($\sigma$-additive real-valued regular) measure on a Tychonoff space $X$. The variation and the norm of $\mu$ are denoted by $|\mu|$ and  $\| \mu\|$ respectively.
We shall use the following well known fact (see for example \cite[7.6.5]{jarchow}).
\begin{fact} \label{f1}
Let $X$ be a Tychonoff space. Then
\begin{enumerate}
\item[{\rm (i)}] $(C_c(X))'$ can be identified with the space of all measures on $X$ with compact support.
\item[{\rm (ii)}] $(C_p(X))'$  can be identified with the space of measures with finite support in $X$.
\end{enumerate}
\end{fact}

If $X$ is  a countable $\aleph_0$-space,  the space $C(X)$ is an $\aleph_0$-space  both in the compact-open and the pointwise topology (see \cite{Mich}). Next theorem  shows that the same holds also for the weak topology on $C_c(X)$.
\begin{theorem}\label{weak}
If $X$ is a countable  $\aleph_0$-space, then $C_c(X)$ is a weakly $\aleph_0$-space.
\end{theorem}

\begin{proof}
Set $E:=C_c(X)$ and let $E_w$ be the space $C_c(X)$ endowed with the weak topology. Let $\DD$ be a countable closed $k$-network in $X$ closed under taking finite unions, and let $\mathcal{B}$ be a countable basis in $\RR$.
For every finite subset $F=\{ x_1,\dots,x_n\}$ of $X$, every finite subfamily $$\mathcal{U}=\{ U_1,\dots,U_n\}$$ of $\mathcal{B}$, each $D\in\DD$ and every $m\in \NN$, set
\begin{equation} \label{e-1}
A(F,\mathcal{U},D,m):=\{ f\in C(X): \; f(x_i)\in U_i, 1\leq i\leq n,\; \mbox{ and } \; f(D)\subset [-m,m] \}.
\end{equation}
Denote by $\mathcal{A}$ the countable family of all subsets of $E$ of the form (\ref{e-1}).
By \cite[Thm. 1]{Guth}, in order to prove the theorem it is enough to show that the family $\mathcal{A}$ satisfies the following claim.

{\bf Claim}. {\it For every $f_0\in E$, for every sequence $\{ f_n\}_{n\in\NN}$ converging to $f_0$ in $E_w$ and any neighborhood $W$ of $f_0$ in $E_w$ there exists $A\in \mathcal{A}$ such that $f_0\in A \subset W$ and $f_n\in A$ for almost all $n\in \NN$.}

Without loss of generality we may assume that $W$ is of the standard form, i.e. there are measures $\mu_1,\dots,\mu_s \in E'$ and $\varepsilon >0$ such that
\[
W = \{ f\in E: |\mu_i (f-f_0) |<\varepsilon, 1\leq i\leq s \}.
\]
Set $K:= \bigcup_{i=1}^s \mathrm{supp}(\mu_i)$. So $K$ is a compact subset of $X$ by Fact \ref{f1}(i).

Let $\{ D'_n\}_{n\in\NN}$ be an enumeration of the family $\{ D'\in\DD : K\subseteq D'\}$. For every $n\in\NN$, set $D_n :=\bigcap_{i=1}^n D'_i$. It follows that the decreasing sequence of sets  $\{ D_n\}_{n\in\NN}$ converges to the compact set $K$ in the sense that each neighborhood $O(K)$ of $K$ contains all but finitely many sets $D_n$.

{\it Step 1}. Let us show that there are $k,m\in\NN$ such that
\begin{equation} \label{1-1}
|f_i(x)| <m, \quad \forall x\in D_m, \; \forall i \geq k.
\end{equation}
Indeed, assuming the converse we choose a sequence $\{ x_n\}_{n\in\NN}$, with $x_n \in D_n$ for every $n\in\NN$, and a sequence $i_1 <i_2<\dots$ such that
\begin{equation} \label{1-2}
|f_{i_n}(x_n)| >n , \quad \forall n\in\NN.
\end{equation}
Since $\{ D_n\}_{n\in\NN}$ converges to the compact set $K$, all accumulation points of the sequence $\{ x_n\}_{n\in\NN}$ are in $K$.
In other words, the set
$$K' := K\cup \{ x_n\}_{n\in\NN}$$ is compact.
As the restriction map $f\mapsto f|_{K'}$ is continuous, we obtain that the sequence $S:= \{ f_{i_n}|_{K'}\}_{n\in\NN}$ in the Banach space $C(K')$ converges to $f_0|_{K'}$ in the weak topology of $C(K')$. Thus $S$ is bounded, that is there is $C>0$ such that
\[
|f_{i_n}(x)| <C, \quad \forall x\in K', \forall n\in\NN.
\]
In particular, $|f_{i_n}(x_n)| <C$ for every $n\in\NN$, that contradicts (\ref{1-2}). This proves (\ref{1-1}).

{\it Step 2}.  Fix $k,m\in\NN$ such that (\ref{1-1}) holds.
For every $1\leq i\leq s$ take a finite subset $F_i$ of $\mathrm{supp}(\mu_i)$ such that
\begin{equation} \label{1-3}
|\mu_i| \left(\mathrm{supp}(\mu_i)\setminus F_i \right) < \frac{\varepsilon}{3m},
\end{equation}
and for every $x_{i,j}\in F_i$ choose $U_{i,j} \in \mathcal{B}$ such that
\begin{equation} \label{1-4}
f_0 (x_{i,j}) \in U_{i,j} \; \mbox{ and } \; \mathrm{diam}(U_{i,j}) < \frac{\varepsilon}{3|F_i| \cdot \| \mu_i\|}.
\end{equation}
If $x_{i,j} =x_{k,l}$ for some $1\leq i<k\leq s$, we shall suppose that $U_{i,j}=U_{k,l}$. Finally we set
\[
A:= \{ f\in C(X): \; f(x_{i,j})\in U_{i,j}, \; x_{i,j}\in F_i, 1\leq i\leq s; \; \; f(D_m)\subset [-m,m] \}.
\]
Clearly, $A\in \mathcal{A}$. For each $f\in A$ and every $1\leq i\leq s$, (\ref{1-3}) and (\ref{1-4})   imply
\[
\begin{split}
|\mu_i (f-f_0)| & \leq \sum_{x_{i,j}\in F_i} |f(x_{i,j})-f_0(x_{i,j})| \cdot \| \mu_i\| \\
& +  \sum_{x_{i,j}\in \, \mathrm{supp}(\mu_i)\setminus F_i} |f(x_{i,j})-f_0(x_{i,j})| \cdot | \mu_i(\{ x_{i,j}\})|\\
& < |F_i| \cdot \frac{\varepsilon}{3|F_i| \cdot \| \mu_i\|} \cdot  \| \mu_i\| + 2m\cdot \frac{\varepsilon}{3m} = \varepsilon.
\end{split}
\]
Thus $A\subset W_0$, and (\ref{1-1}) shows that  $f_i (D_m)\subset [-m,m]$ for every $i\geq k$. Since $f_n \to f_0$ also in the pointwise topology, we obtain that $f_n\in A$ for all sufficiently large $n\in\NN$. This proves Claim and hence also the theorem.
\end{proof}
Consequently, the space $C_{c}(X)$ is a weakly $\aleph_{0}$-space for any metrizable and countable space $X$.
We end this section with the following conjecture
\begin{conjecture}
Let $X$ be a Tychonoff space. Then $C_c(X)$ is a weakly $\aleph$-space if and only if $C_c(X)$ is a weakly $\aleph_0$-space if and only if $X$ is a countable $\aleph_0$-space.
\end{conjecture}

\section{One application}

Let $E$ be a separable Banach space and $S$ its closed unit ball endowed with the weak topology of $E$. In \cite[Theorem A]{edgar} Edgar and Wheller proved that $S$ is completely metrizable if and only if $S$ is a Polish space if and only if $S$ is metrizable and every closed subset of $S$ is a Baire space. We supplement this result.

For this purpose we introduce a property stronger than the property to be an $\aleph_0$-space by \cite[Theorem 11.4]{Mich}. We say that a topological space $X$ is an {\it $\aleph_1$-space} if $X$ is a continuous image under a compact-covering map from a Polish space $Y$. Every closed subspace of an $\aleph_1$-space is also an $\aleph_1$-space.

\begin{proposition}\label{latest}
 Let $E$ be a separable Banach space.
\begin{enumerate}
\item[{\rm (i)}] If  $E$ does not contain a copy of $\ell_{1}$, then the closed unit ball $S$ of $E$ is a Polish space in the weak topology of $E$ if and only if $E$ is a weakly $\aleph_{1}$-space.
\item[{\rm (ii)}] If $E$ contains a copy of $c_{0}$, then $E$ is not a weakly $\aleph_{1}$-space.
\end{enumerate}
\end{proposition}
\begin{proof}
(i) The closed unit ball $S$ endowed with the weak topology we denote by $S_w$.
Assume that $S_w$ is a Polish space. For each $n\in\mathbb{N}$ set $S_{n}:=nS_w$ and let $Y:=\bigoplus_{n}S_{n}$ be the topological direct sum of the sequence $(S_{n})_{n}$ of Polish spaces. Denote by $T$  the canonical mapping from $Y$ onto $E_w :=(E,\sigma(E,E'))$. Since every compact set of $E_w$ is contained in some $S_{m}$, the map $T$ is continuous and compact-covering. Conversely, assume that $E$ is a weakly $\aleph_{1}$-space and $T:Y\rightarrow E_w$ is a continuous compact-covering map. Denote by $B(x,r)$ the closed ball in $Y$ of radius $r$ centered at $x$. For a countable dense sequence $(x_{j})_{j\in\NN}$ in $Y$ and each $\alpha=(n_{k})\in\mathbb{N}^\mathbb{N}$, set $K_{\alpha}:=\bigcap_{k=1}^{\infty}\bigcup_{j=1}^{n_{k}}B(x_{j},k^{-1})$.
Then $\{K_{\alpha}:\alpha\in\mathbb{N}^\mathbb{N}\}$ is a family of compact sets in $Y$ covering $Y$ with  $K_{\alpha}\subset K_{\beta}$ whenever $\alpha\leq\beta$, and such that every compact set in $Y$ is contained in some $K_{\alpha}$. Set $W_{\alpha}:=T(K_{\alpha})$ for each $\alpha\in\NN^\NN$. Since $T$ is compact-covering and continuous, the sets $W_{\alpha}$ compose a  compact covering of $E_w$ such that every $\sigma(E,E')$-compact set is contained in some $W_{\alpha}$. On the other hand, $E_w$ is an $\aleph_{0}$-space, so we apply
Corollary \ref{c-Equiv} to deduce that the strong dual $E'$ is separable. Hence $S_w$ is metrizable (see Lemma \ref{sep}) and separable. Now by Christensen's theorem, see \cite[Theorem 6.1]{kak}, the space $S_w$ is a Polish space.

(ii) The closed unit ball $B$ of $c_{0}$ is  metrizable and separable in the weak topology. On the other hand, by  \cite[Theorem A, Examples (3)]{edgar} $B$ is  not a Polish space in the weak topology.  Now the proof of (i) (involving the Christensen's theorem) applies to complete the case (ii).
\end{proof}
\begin{remark}\label{latest-1} \em{
Note that $\ell_{1}$ is a weakly $\aleph_{1}$-space (by the Schur property) but the unit ball in $\ell_{1}$ is not a Polish space, see \cite[Example 9]{edgar}. So the assumption on $E$ in item (i) that $E$ does not contain a copy of $\ell_1$ is essential. Recall also that for a Banach space $E$ with separable bidual $E''$ the unit ball $S$ in $E$ is a Polish space in the weak topology by Godefroy's theorem, see \cite[Theorem 12.55]{fabian}.}
\end{remark}
\begin{remark}\label{latest-2} \em{
Let $K$ be a countably infinite compact space. Then $C(K)$ contains a copy of $c_{0}$. Hence $C(K)$ is not a weakly $\aleph_{1}$-space by (ii),  but $C(K)$  is a weakly $\aleph_0$-space by Theorem \ref{tMain}.}
\end{remark}

\bibliographystyle{amsplain}

\end{document}